\documentclass[10pt,leqno,a4paper]{article}
\usepackage{amsmath}
\usepackage{amsthm}
\usepackage{amssymb}
\usepackage[hypertex]{hyperref}
\usepackage[english]{babel}
\usepackage[totalheight=22 true cm, totalwidth=14 true cm]{geometry}
\usepackage{setspace}
\usepackage{mathrsfs}
\usepackage[all]{xy}
\newtheorem{thm}{Theorem}[section]
\newtheorem{cor}[thm]{Corollary}
\newtheorem{lemma}[thm]{Lemma}

\newtheorem{prop}[thm]{Proposition}

\newtheorem{defn}[thm]{Definition}

\theoremstyle{definition}

\newtheorem{parag}[thm]{}
\newtheorem{rmk}[thm]{Remark}
\newtheorem{exa}[thm]{Example}

\theoremstyle{remark}

\numberwithin{equation}{thm}

\def\beq{\begin{equation}}
\def\eeq{\end{equation}}

\def\ben{\begin{enumerate}}
\def\een{\end{enumerate}}

\def\crash#1{}

\def\Pan{{\mathbf P}}
\def\Aan{{\mathbf A}}
\def\Q{{\mathbb Q}}
\def\R{{\mathbb R}}

\def\l{\left}
\def\r{\right}
\def\[[{\l[\l[}
\def\]]{\r]\r]}
\def\p{\prime}

\def\ord{{\rm ord}}
\def\an{{\rm an}\;}

\def\pr{{\rm pr}}

\def\cf{\emph{cf.}\;}
\def\ie{\emph{i.e.}\;}
\def\lc{\emph{loc.cit.}\;}

\def\cB{{\mathcal B}}

\def\cE{{\mathcal E}}
\def\cF{{\mathcal F}}

\def\cM{{\mathcal M}}

\def\cO{{\mathcal O}}

\def\cJ{{\mathcal J}}

\def\cL{{\mathcal L}}

\def\cR{{\mathcal R}}
\def\cS{{\mathcal S}}

\def\sH{{\mathscr H}}

\def\sY{{\mathscr Y}}

\def\fC{{\mathfrak C}}

\def\fF{{\mathfrak F}}

\def\fX{{\mathfrak X}}
\def\fY{{\mathfrak Y}}
\def\fZ{{\mathfrak Z}}

\def\what{\widehat}

\def\veps{\varepsilon}

\def\an{{\rm an}}

\def\sp{{\rm sp}}

\def\Spf{{\rm Spf\,}}

\def\ul{\underline}
\def\ol{\overline}
\def\iso{\xrightarrow{\ \sim\ }}

\def\kc{{k^\circ}}

\def\kt{\widetilde{k}}
\def\wt{\what{\otimes}}

\author{by \\  \\ Francesco Baldassarri\thanks{Universit\`{a} di Padova,
Dipartimento di matematica pura e applicata, Via Trieste, 63, 35121 Padova, Italy.}}

\title{Radius of convergence of $p$-adic connections \\
and  the Berkovich ramification locus}

\begin{document}

\maketitle

\begin{abstract} 
We apply  the theory of  the radius of convergence of a $p$-adic connection \cite{Cont} 
to the special case of the direct image of the constant connection via a finite morphism of 
compact $p$-adic curves, smooth in the sense of rigid geometry. In the case of an \'etale covering of curves with good reduction, we get a  
lower bound for that radius, corollary \ref{appltriv},  and obtain, via corollary \ref{Newton}, a new geometric  proof of a variant of 
 the $p$-adic Rolle theorem of  Robert and Berkovich, theorem \ref{RolleAN}. 
% The proof is based on a widely believed, although unpublished,
% result of simultaneous semistable reduction for finite morphisms of smooth $p$-adic curves. We deduce from 
% it a useful description, theorem \ref{galois} and its corollary \ref{covdisk},
% of the Galois structure of \'etale coverings of smooth curves, 
% which is new and 
% of independent interest.
 We  take this opportunity to clarify the relation between the notion of radius of convergence used in  \cite{Cont}  and the more intrinsic one used by Kedlaya \cite[Def. 9.4.7]{Ke2}.
 \end{abstract}

\tableofcontents

\medskip
%%%%%%%%%%%%%%%%%%%%%%%%%%%%%%%%%%%%%%%%%%%%%%%%%%%%%%%%%%%%%%%%%%%%
\setcounter{section}{-1}
\section{A geometric $p$-adic Rolle theorem}
Let $(k,|~ |)$ be a complete algebraically closed  extension of $(\Q_{p},|~|_{p})$, with  $|p|_{p} =p^{-1}$, and let $\kc$ be the ring of integers of $k$.  We consider 
$k$-analytic spaces in the sense of Berkovich. 
We want to illustrate our  theory of the radius of convergence of a $p$-adic connection \cite{Cont}, by deducing from it a conceptual proof  of a global form of 
 the $p$-adic Rolle theorem of  Robert  \cite[\S 2.4]{Robert}, \cite[Prop. A.20]{Potential} in the stronger form due to  Berkovich  (corollary \ref{BerkRolle} below). Our result 
% is equivalent to Robert's, but 
 indicates a new approach to the problem and, in favorable global situations, offers a finer geometric understanding.
 \par
 
Let $\varphi : Y \to X$ be a morphism of smooth  $k$-analytic curves. If $\varphi$ is \'etale at a $k$-valued point  $y \in Y(k)$, then, as in the familiar complex case, $\varphi$ induces an open embedding $\varphi_{|U}: U \hookrightarrow X$, of an open neighborhood $U$ of $y$, in $X$. But,  as is rather the case in algebraic geometry, this property may fail
at a more general type of point $y \in Y$, even if  $\varphi$ is \'etale at $y$ \cite{DR2}.  

\begin{defn} \label{critical}  Let $\varphi : Y \to X$ be a morphism of rig-smooth  strictly $k$-analytic curves. We say that $y \in Y(k)$ is  a \emph{critical point} of $\varphi$ if 
the differential $d\varphi_{y} = 0$, \ie if $\varphi$ not rig-\'etale at $y$. We denote by 
${\rm Crit} (\varphi)$ the set of critical points of $\varphi$ and let
  $B(\varphi) = \varphi( {\rm Crit}(\varphi)) \subset X(k)$ be the \emph{classical branch locus} of $\varphi$, and $Z(\varphi) = \varphi^{-1}(B(\varphi)) \subset Y(k)$  be the  (saturation of the)  \emph{classical ramification locus} of $\varphi$. We define the  \emph{Berkovich ramification locus} $\cR_{\varphi}$ of $\varphi$ as the  set of points of $Y$ at which $\varphi$ is not a local open embedding. 
\end{defn}
So,   $\cR_{\varphi}$ is a  closed subset of $Y$ and 
$\cR_{\varphi} \cap \, Y(k) = {\rm Crit} (\varphi)$. We are interested in bounding from below the distance  of $\cR_{\varphi}$ from a non-critical $k$-valued point $y \in Y(k)$, in the case of a finite morphism $\varphi$, as above,  of compact curves.
  
\par
Our interest in this topic arose from reading Faber's  papers \cite{Xa1} \cite{Xa2}, where this question is answered, via explicit computations, for a non-constant rational function $\varphi$, viewed as a finite flat map $\varphi  : \Pan \to \Pan$, of the $k$-analytic projective line $\Pan$ to itself. The novelty in Faber's paper  concerns   the case of an open disk $D \subset \Pan$, with $D \cap {\rm Crit} (\varphi) = \emptyset$,  such that  $\varphi(D) = \Pan$, a case which cannot be deduced from the classical statement. 

We cannot  prove  Faber's result by our method.  
We prove instead
\begin{thm} \label{RolleAN} Let $\varphi: Y \to X$ be an \'etale covering  of compact rig-smooth strictly $k$-analytic curves with good reduction.  Let  $D \subset Y$  be any open disk equipped with a \emph{normalized coordinate} $T: D \iso D(0,1^{-})$.
 Then, 
 for any open disk $D^{\p} \subset D$ of  $T$-radius $\leq p^{-\frac{1}{p-1}}$, $\varphi$ induces an open embedding  $D^{\p}  \to X$. 
% Moreover, if $\varphi$ is residually separable at the boundary point $\zeta$ of $D$ in $Y$, then $\varphi$ induces an open embedding  $D  \to X$ of $D$ itself. 
 \end{thm}
 The following is Berkovich generalization of the $p$-adic Rolle theorem of  \cite[\S 2.4]{Robert}.
 \begin{cor} (Berkovich)  \label{BerkRolle}Let $\varphi : D(0,1^{-}) \to {\bf A}$ be any \'etale morphism of the open unit disk to the $k$-analytic affine line $\bf A$, then the restriction of $\varphi$ to any open disk of radius  $p^{-\frac{1}{p-1}}$ is an open embedding.    
 \end{cor}
 \begin{proof}
The statement follows from the theorem, since the restriction of $\varphi$ to any strictly affinoid disk $E \subset D$ induces a finite map $\varphi_{|E}: E \to \varphi(E)$ to which the theorem applies. 
 \end{proof}
 \begin{cor} \label{RolleANcor} Let $\varphi: Y \to X$ be a finite morphism of compact rig-smooth strictly $k$-analytic curves.  Let  $D \subset Y \setminus Z(\varphi)$  be any open disk with $T$  a  normalized coordinate on $D$. 
  Then, the $T$-distance of $\cR_{\varphi} \cap D$ from $D(k)$ is $\geq p^{-\frac{1}{p-1}}$ \end{cor}
 \begin{proof}
The statement may be deduced from (\ref{BerkRolle}), as follows. We assume, with no loss of generality, that $X$ and $Y$ are connected. Suppose first that $X$ is \emph{projective}. If  the genus of $X$ is $\geq 1$, it follows from \cite[4.5.3]{Berkovich} that $\varphi(D)$ is contained in an open disk contained in $X$. So,  the classical theorem applies.  The case when $X$ is the projective line and $\varphi_{|D}$ is not surjective, is covered by the classical  theorem, too.  If $\varphi(D) = X = \Pan$, then, $\varphi$ being finite,  the $p$-adic GAGA  implies that $Y$ is projective as well. 
 But then the assumption on the branch locus is only verified if $\varphi$ is an isomorphism, which contradicts $\varphi(D) = \Pan$. 
  Now, (\cf \cite[3.2]{dejong}) a  compact rig-smooth curve is either affinoid or projective.  But we know  \cite[1.2.5]{Cont} that, if $X$ is affinoid, then it is an affinoid domain in a connected projective curve $C$, \emph{formal} with respect to a strictly semistable model $\fC$ of $C$. So, again, \cite[4.5.3]{Berkovich} 
 shows that the only case not covered by the classical theorem is when  $X = \Pan$ and $\varphi$ is surjective, which leads us back to the former discussion. 
 \end{proof}
 So, the three statements above (\ref{RolleAN}), (\ref{BerkRolle}) and (\ref{RolleANcor}), are equivalent. A beautifully simple analytic proof of (\ref{BerkRolle}) has been communicated to the author by Professor Berkovich.  We reproduce it here with his permission.
\begin{proof} (Berkovich proof of (\ref{BerkRolle}))
Any morphism $\varphi: D = D(0,1^{-}) \to {\bf A}$ is a formal power series $\varphi = \sum_{n=0}^{\infty} a_{n}T^{n}$, convergent on $D$.
One can easily see that
\ben
\item
 $\varphi$ is \'etale if and only if the derivative $\frac{d \varphi}{dT} = \sum_{n=1}^{\infty} n a_{n}T^{n -1}$
 is invertible  on $D$, \ie,  for every
$0 < \rho <1$, one has $|a_{1}| > |n a_{n}| \rho^{n-1}$,  for all $n \geq 2$;
 \item
$\varphi$  induces  an open immersion $D(0,r^{-}) \hookrightarrow {\bf A}$ if and only if, for every $0<\rho<r$, one has $|a_{1}|> |a_{n}| \rho^{n-1}$ for all
$n \geq 2$.
 \een

The claim is equivalent to the following simple fact for $\varphi$ as above: if 
 $|a_{1}| > |n a_{n}| \rho^{n-1}$ 
for all $n \geq 2$, then 
 $|a_{1}| > |a_{n}| (\rho |p|^{\frac{1}{p-1}})^{n-1}$  for all $n \geq 2$. In its turn, this fact it is equivalent to
following inequality: 
 $|n| \geq |p|^{\frac{n-1}{p-1}}$, for all $n \geq 2$. 
 If $n$ is not divisible by $p$, the latter inequality is trivial. Suppose that $n = p^{k}m$ with $m$ prime
to $p$ and $k \geq 1$. Then the latter inequality is equivalent to the inequality 
 $p^{k}m -1 \geq k(p-1)$. One has 
 $$p^{k}m -1 \geq p^{k} -1 = (p^{k-1} + \dots+ 1)(p-1) \geq k(p-1) \; ,$$
   and the claim  follows.
\end{proof}
We hope however  that our proof of  (\ref{RolleAN}) in section \ref{HEART} below, based on the theory of the radius of convergence of $p$-adic connections, will be instructive. 
 \begin{rmk} \label{disks}
  A  $k$-analytic curve $E$ is an  \emph{open disk} if it is  isomorphic to the open analytic domain 
$$ D(0,r^{-}) =  \{ x \in \Aan | |T(x)| <r \}
\; ,
$$
of  the $k$-analytic $T$-line $\Aan$, for some $r \in \R_{>0}$.   
The only isomorphism invariant of $E$ is then the image of $r$ in $\R_{>0}/|k^{\times}|$. We say that the open disk $E$ is  \emph{strict},  if  $r \in |k^{\times}|$, \ie if  $E$ is isomorphic to the standard unit open disk $D(0,1^{-})$.   An open disk which is a relatively compact analytic domain in a $k$-analytic curve $X$, is strict (resp. non strict) if and only if its boundary point in $X$ is of type 2 (resp. 3). This topic will be clarified in \cite{ducros}. In particular, the point $\zeta \in Y$ at the boundary of $D$ in the statements \ref{RolleAN} and \ref{RolleANcor}, is a point of type 2. We also use the notation
$$ D(0,r^{+}) =  \{ x \in \Aan | |T(x)| \leq r \}
\; ,
$$
for the standard closed disk in $\Aan$, and $\alpha_{0,r}$ for its maximal point. 
\end{rmk}
% \begin{rmk} \label{ressep} Let $\xi = \varphi(\zeta)$ be as above. The map $\varphi : Y \to X$, $\zeta \mapsto \xi$, induces an isometric embedding $\sH(\xi) \subset \sH(\zeta)$, hence a $\kt$-linear embedding $\wtilde{\sH(\xi)} \subset \wtilde{\sH(\zeta)}$. Then $\xi$ is a point of type 2, as well. The degree of inseparability $[ \wtilde{\sH(\zeta)}: \wtilde{\sH(\xi)}]_{i}$, is a power of $p$, called the \emph{residual inseparability of $\varphi$ at $\zeta$}. Then $\varphi$ is \emph{residually separable at $\zeta$} if its residual inseparability at $\zeta$ is 1. In particular, if $\varphi$ is tame at $\zeta$ \cite[6.3]{BerkovichEtale}, then it is residually separable at $\zeta$.
%   \end{rmk}
 \begin{rmk} \label{height} The $T$-radius of $D^{\p}$ appearing in the theorem, will be called the \emph{relative radius of $D^{\p}$ in $D$} or the \emph{height}  of the semi-open annulus $D \setminus D^{\p}$. It is an analytic invariant $0< h(D \setminus D^{\p}) \leq 1$ of  $D \setminus D^{\p}$ as in \cite[\S 2]{BL0}.
 \end{rmk}
 \begin{rmk} \label{imagedisk} Let $\varphi : D(0,1^{-}) \to D(0,1^{-})$ be a morphism of the open unit $k$-disk to itself, such that $\varphi (0) = 0$, while $\varphi (D(0,1^{-}))$ is not contained in  $D(0,\rho^{-})$, for any $\rho <1$. Then $\varphi$ is surjective.  This follows from the elementary theory of Newton polygons.  In fact, in the standard coordinate $T$, $\varphi$ is represented by a power series $\varphi(T) = \sum_{i=r}^{\infty} a_{i} T^{i}$, with $a_{i} \in \kc$, such that $a_{r} \not= 0$ for some $r \geq 1$, and $\inf_{i} \ord_{p} a_{i} = 0$. So, for any $a \in D(0,1^{-})$, $\ord_{p} a > 0$,  the Newton polygon of $\varphi(T) - a$ has a side with negative slope $-\sigma$, where 
 $$\sigma := \frac{\ord_{p} a - \ord_{p} a_{r+j}}{r+j} > 0 \; ,$$
for some $j \geq 0$.  See Fig. 1. 

 \begin{figure}[ht]
\begin{picture}(330,170)(-50,0)      
\put(0,10){\vector(1,0){330}}        
\put(25,0){\vector(0,1){150}}         
\put(65,50){\line(4,-1){40}}
\thicklines               
\put(0,60){\makebox(0,0){$(0,\ord_{p} a)$}}
\put(65,70){\makebox(0,0){$\ast$}}
\put(95,100){\makebox(0,0){$\ast$}}
\put(265,20){\makebox(0,0){$\ast$}}
\put(295,18){\makebox(0,0){$\ast$}}
\put(235,45){\makebox(0,0){$(i,\ord_{p} a_{i})$}}
\put(92,73){\makebox(0,0){$(r,\ord_{p} a_{r})$}}
\put(65,70){\line(3,-2){45}}
\put(65,70){\line(0,1){100}}
\put(109,39){\makebox(0,0){$\ast$}}
\put(209,39){\makebox(0,0){$\ast$}}
\put(200,173){\makebox(0,0)[b]{}}
\put(25,60){\makebox(0,0){$\ast$}}
\put(25,60){\line(4,-1){220}}
\put(147,45){\makebox(0,0){$(r+j,\ord_{p} a_{r+j})$}}
\put(109,40){\line(5,-1){75}}  
\put(255,18){\line(1,0){75}}  
\put(80,25){\makebox(0,0){slope = $- \sigma$}}
\end{picture}
\caption{\label{fig 0} The Newton polygon of $\varphi (T) - a$}
\end{figure}
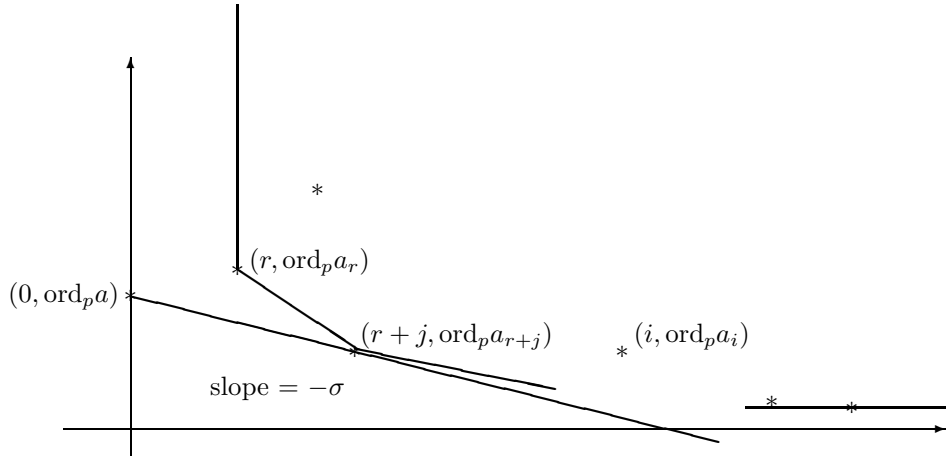
 \end{rmk}
 \begin{rmk}
 \label{Faber} The rational function $\varphi(T) = \frac{T^{p+1} - p}{T}$ restricts to a surjective \'etale map $D(0,1^{-}) \to \Pan$ to which the classical theorem does not apply, but Faber's does.  This is example 5.3 in  \cite{Xa2}. Notice that $0 \in Z(\varphi) \setminus  {\rm Crit}(\varphi)$, so that our statement does not apply. 
  \end{rmk}  
Our proof is based on the most basic result on $p$-adic differential systems, namely the so-called \emph{trivial estimate} for the radius of convergence of their solutions \cite[p. 94]{DGS} and on the \emph{transfer principle} into a disk with no singularity  \cite[IV.5]{DGS}.
%, once a certain integrality result is established. We deduce this integrality statement from  a result on simultaneous semistable reduction of $k$-analytic curves, (\ref{extension}) below, due to Coleman \cite{Coleman} and improved by Liu \cite{Liu2} in the projective case. See also Temkin \cite{temkin}. This result  is apparently well-known to specialists, but, as far as we know, unpublished. \par
%We deduce from our integrality statement the following result, of independent general interest.
%\begin{thm} \label{galois} Let $\varphi: Y \to X$ be a finite morphism of connected compact rig-smooth strictly $k$-analytic curves and let $D \subset Y$ be a strict open disk such that  $\varphi(D)$ is not compact. Then $\varphi (D) =: E$ is a strict open disk in $X$.  
%Assume the induced covering $\varphi^{-1}(E) \to E$ is \'etale. Then $\varphi^{-1}(E) \to E$ is Galois. In particular, 
%$$\varphi^{-1}(E) = \dot\bigcup_{i=1}^{r} D_{i}
%\;,$$
% with $D = D_{1}$, a disjoint union of open disks isomorphic to $D$, such that the induced covering $\varphi_{i}:D_{i} \to E$ is Galois with group the stabilizer of $D_{i}$.
%\end{thm}
\begin{rmk} \label{compactim} As we observed before, in the situation of (\ref{RolleANcor}), if $\varphi(D)$ is compact,  then $\varphi(D) = X  \cong \Pan$, and  therefore, if $\varphi$ is unramified, it is an isomorphism, so that $\cR_{\varphi} = \emptyset$.
\end{rmk}
\begin{rmk} \label{noncompactim} We observed in the proof of (\ref{RolleAN}) that if $\varphi(D)$ is not compact,  then 
it is contained in a strict open disk in $X$. From remark \ref{imagedisk} it then follows that $\varphi(D) = E$ is a strict open disk in $X$. 
\end{rmk}
%We deduce from (\ref{galois}) the following remarkable consequence. 
%\begin{cor} \label{covdisk} Any finite \'etale covering of an open disk by an open disk is Galois.
%\end{cor}
%\begin{proof} Suppose $\varphi : D(0,1^{-}) \to D(0,1^{-})$ is a finite \'etale covering of degree $d$, such that $\varphi(0) = 0$. It will suffice to prove that
% for any $\rho \in |k| \cap (0,1)$, the restriction of $\varphi$ to $\varphi_{\rho}: \varphi^{-1}(\varphi(D(0,\rho^{-})) \to D(0,\rho^{-})$ is Galois. Let 
% $$R(\rho) = \sup \{r \in (0,1) | \varphi(D(0,r^{-})) \subset D(0,\rho^{-}) \} \; .
% $$ 
% We have $R(\rho) <1$, because $\varphi (D(0,1^{-})) = D(0,1^{-})$. Moreover, since $\varphi(D(0,r^{-}))$ is an open disk in $D(0,1^{-})$, as we saw before, we have $\varphi(\alpha_{0, R(\rho)}) = \alpha_{0,\rho}$. In particular, $\alpha_{0, R(\rho)}$ is a point of type 2, so that $R(\rho) \in |k| \cap (0,1)$ and $\varphi(D(0,R(\rho)^{-}) = D(0,\rho^{-})$. 
% Let
% $\varphi^{+}_{\rho}: \varphi^{-1}(\varphi(D(0,\rho^{+})) \to D(0,\rho^{+})$ be the restriction of $\varphi$.  This map, together with the inclusion 
% $D(0,R(\rho)^{-}) \subset \varphi^{-1}(\varphi(D(0,\rho^{+}))$ ($\subset D(0,1^{-})$), satisfies the assumptions of theorem \ref{galois}. We conclude that 
% $\varphi_{\rho}: \varphi^{-1}(\varphi(D(0,\rho^{-})) \to D(0,\rho^{-})$ is Galois.
%\end{proof}
\par
Although not really needed for the conclusion of our proof (a reference to either one of \cite{Cont} or \cite{Ke2}, independently, would suffice),
we   recall  in the last section  \label{Graphs}  the main properties of the radius of convergence of a connection on a compact rig-smooth $p$-adic analytic curve $X$ with poles at a finite subset  $Z \subset X(k)$ \cite{Cont}. 
In that paper,  we consider  a (sufficiently fine) strictly semistable $\kc$-formal model $\fX$ 
of $X$ and an extension $\fZ$ of $Z$ to a divisor of the smooth part of $\fX$, \'etale over $\kc$. We  then introduce a 
\emph{global} notion of \emph{$(\fX,\fZ)$-normalized radius of convergence  $\cR_{\fX,\fZ}(x,(\cM,\nabla))$  of  $(\cM,\nabla) \in {\bf MIC}((X \setminus Z)/k)$ at $x \in X \setminus Z$}. (In case $\fZ = \emptyset$, we simply write  $\cR_{\fX}$ instead of $ \cR_{\fX,\emptyset}$). 
 We take this opportunity to completely clarify the relation between $\cR_{\fX,\fZ}(x,(\cM,\nabla))$ and  the \emph{local} notion of \emph{intrinsic generic radius of convergence $IR(\cM_{(x)},\nabla)$ of  $(\cM,\nabla)$ at $x$}, for a point $x \in X$ of Berkovich type 2 or 3, used by  Kedlaya \cite[Def. 9.4.7]{Ke2}. 
 The coincidence of the two notions when $x$ is a point of the skeleton $\Gamma_{\fX,\fZ} \setminus Z$, will be useful in general. 
 It is here only used implicitly (in an obvious case)  in the conclusion of our proof.
\par
I am indebted to V. Berkovich  and to X. Faber for their explanations on the $p$-adic Rolle theorem and to M. Temkin for pointing out a mistake in a previous version of this paper. Discussions with  Y. Andr\'e, P. Berthelot, M. Cailotto,  L. Illusie,  Q. Liu,  M. Raynaud have been most useful: I thank them heartly for that. It is a pleasure to aknowledge the well-founded criticism and  invaluable suggestions provided by the referee. 
 \section{A change in viewpoint}
 
 \parag \label{notation}
  Let $\varphi: Y \to X$ be a finite morphism of compact connected rig-smooth strictly $k$-analytic curves. 
  We use by default (strictly) $k$-analytic spaces in the sense of Berkovich \cite{BerkovichEtale} endowed with their natural topology.  
%  We assume without loss of generality that $X$ and $Y$ are connected strictly $k$-analytic curves. 
  By  (a minor variation of) \cite[Cor. 3.4]{dejong}, they are good strictly $k$-analytic  spaces. Now, a finite morphism of good analytic spaces is good and closed, so that proposition 3.2.9 of \cite{BerkovichEtale}  applies, and $\varphi$ is finite flat and, in particular, open. It follows that, for any strict affinoid $U \subset X$, $\varphi^{-1}U$ is affinoid and $\varphi_{\ast}\cO_{\varphi^{-1}U}$ is locally free for the Zariski topology of $U$. Then, $\varphi$ being good, it follows that  $\varphi_{\ast}\cO_{Y}$ is locally free for the natural topology of $X$. Since
 $X$ is connected, the degree of $\varphi$ is a constant $d$ on $X$. 
Moreover, $\varphi$ is generically \'etale (\ie \'etale but at a discrete set of $k$-rational points) because we are in characteristic zero.  We recall that an irreducible  compact $k$-analytic curve is either the analytification of a projective curve or it is  affinoid  \cite{FM}, \cite[Prop. 3.2]{dejong}.  
 \parag
Let $B = B(\varphi) \subset X(k)$ and $Z = Z(\varphi)$ be as before. Let $\cJ_{B}$ (resp. $\cJ_{Z}$) denote the ideal sheaf of $B$ (resp. $Z$). 
For a coherent $\cO_{X}$-module (resp. $\cO_{Y}$-module) $\cF$, we denote by $\cF(\ast B)$ (resp. $\cF(\ast Z)$) the union 
$\bigcup_{N \geq 1}\cJ_{B}^{-N} \otimes \cF$ (resp.  $\bigcup_{N \geq 1}\cJ_{Z}^{-N} \otimes \cF$).
 
The map $\varphi$ restricts to an \'etale covering $Y \setminus Z \to X \setminus B$ of degree $d$. 
Hence, $\Omega^{1}_{Y  \setminus Z} = \varphi^{\ast} \Omega^{1}_{X \setminus B}$ and $\varphi_{\ast }\Omega^{1}_{Y \setminus Z} =  \varphi_{\ast} \cO_{Y \setminus Z} \otimes_{\cO_{X \setminus B}}\Omega^{1}_{X \setminus B}$, by the projection formula.  More precisely,  
$\Omega^{1}_Y \subset \Omega^{1}_Y(\ast Z) = \varphi^{\ast} ( \Omega^{1}_{X}(\ast B))$ and $\varphi_{\ast } (\Omega^{1}_{Y} (\ast Z)) =  \varphi_{\ast} \cO_Y \otimes_{\cO_X}\Omega^{1}_{X}(\ast B)$. 
 The direct image 
 
\beq \label{projformula}
 \varphi_{\ast}(d_{Y/k}: \cO_{Y} \to \Omega^{1}_{Y}(\ast Z))\;\;\; = \;\;\; \varphi_{\ast}(d_{Y/k}) :  \varphi_{\ast} \cO_{Y} \to  \varphi_{\ast} \cO_{Y}  \otimes \Omega^{1}_X(\ast B) \; ,
\eeq
is then a connection on the locally free $\cO_{X}$-module $\cF := \varphi_{\ast} \cO_{Y}$ of rank $d$, with poles at $B$. We denote by ${\bf MIC}(X(\ast B)/k)$ the tannakian category of such objects, so 
 that  
 \beq
 \label{dirconn}
 (\cF,\nabla_{\cF}) := (\varphi_{\ast} \cO_{Y}, \varphi_{\ast}(d_{Y/k})) \in {\bf MIC}(X(\ast B)/k) \; .
\eeq
 Similarly, 
 \beq
 \label{invconn}
 (\cE,\nabla_{\cE}) := \varphi^{\ast}(\cF,\nabla_{\cF}) = (\varphi^{\ast}\varphi_{\ast} \cO_{Y}, \varphi^{\ast} \varphi_{\ast}(d_{Y/k})) \in {\bf MIC}(Y(\ast Z)/k) \; .
\eeq 
  \begin{parag}
 A more precise version of the formula \ref{projformula} is obtained if we view the pairs $(Y,Z)$, $(X,B)$ as smooth log-schemes over the log-field $(k,k^{\times})$ \cite{Kato}. The analytic map  $\varphi$ induces in fact a finite log-\'etale morphism $\varphi : (Y,Z) \to (X,B)$, locally free of degree $d$, so that 
 $$\Omega^{1}_{Y}(\log Z)  =\varphi^{\ast} \Omega^{1}_{X}(\log B) \; .$$
Therefore formula \ref{dirconn} admits the refinement
\beq
\label{natconnlog}
\varphi_{\ast}(d_{Y/k} : \cO_{Y}  \to \Omega^{1}_{Y}(\log Z))  =  \nabla_{\cF} : \cF  \to \cF  \otimes_{\cO_{X}} \Omega^{1}_{X}(\log B) \; ,
\eeq
which shows that the natural $X/k$-connection with poles along $B$ on the locally free $\cO_X$-module of rank $d$,  $\cF = \varphi_{\ast} \cO_{Y}$, admits logarithmic singularities along $B$. 
Similarly for formula \ref{invconn}, where
 \beq
 \label{natinvconn}
\nabla_{\cE} : \cE  \to \cE  \otimes_{\cO_{Y}} \Omega^{1}_{Y}(\log Z) \; .
\eeq 
\end{parag} 

 \begin{parag} 
  Notice that $\varphi_{\ast} \cO_{Y}$ is also a sheaf of  commutative $\cO_{X}$-algebras and that the multiplication map
 \beq
 \mu_{Y}: \varphi_{\ast} \cO_{Y} \otimes_{\cO_{X}} \varphi_{\ast} \cO_{Y} \longrightarrow \varphi_{\ast} \cO_{Y}
 \eeq
 is horizontal.
 We define two sheaves on $X$. The first
 \beq
 {\cS}ect(Y/X) :=  {\sH}om_{\cO_{X}-{\rm alg}} ({\varphi}_{\ast} \cO_{Y} , \cO_{X}) \; ,
 \eeq
 is a sheaf of finite sets of cardinality $\leq d$.  It is  the \emph{sheaf of local sections} of $Y/X$ or of $\varphi$.  In fact, for any affinoid $U \subset X$, $V := \varphi^{-1}(U) \subset Y$ is an affinoid domain as well, and 
 \begin{eqnarray*}
 Hom_{X}(U,Y) = Hom_{U}(U,V) = Hom_{\cO(U)-{\rm alg}}(\cO(V),\cO(U)) \\
 \\
 = \Gamma(U, {\sH}om_{\cO_{X}-{\rm alg}} ({\varphi}_{\ast} \cO_{Y} , \cO_{X}) ) = \Gamma(U,  {\cS}ect(Y/X) ) \; .
 \end{eqnarray*}
 The second is the  sheaf of $k$-vector spaces of dimension $\leq d$
 \beq
  {\cS}ol(\cF,\nabla_{\cF}) := 
 {\sH}om_{\cO_{X}} ((\cF,\nabla_{\cF}),(\cO_{X},d_{X/k}))^{\nabla} \; ,
 \eeq
 called the \emph{sheaf of local solutions} of $(\cF,\nabla_{\cF})$.  Notice that for any $x_{0} \in X(k) \setminus B$, there exists an open neighborhood $U$ of $x_{0}$, such that ${\cS}ect(Y/X)_{|U}$ is the constant sheaf $\{1,\dots,d\}$ and that $ {\cS}ol(\cF,\nabla_{\cF})_{|U}$ is a $k$-local system of rank $d$.
\end{parag}

The crucial remark is 
\begin{lemma} \label{sections} We have an inclusion of sheaves of sets 
 \beq
 \label{sectionform}
{\cS}ect(Y/X)   \subset   {\cS}ol(\cF,\nabla_{\cF}) \; .
 \eeq
 For any $x \in X \setminus B$, ${\cS}ect(Y/X)  _{x}$ is a $k$-basis of $ {\cS}ol(\cF,\nabla_{\cF})_{x}$, \ie 
 the sheaf of $k$-vector spaces ${\cS}ol(\cF,\nabla_{\cF})$ is  freely generated by its subsheaf  ${\cS}ect(Y/X)$.
\end{lemma}
\begin{proof} We  observe that 
the construction 
\beq
(\varphi : Y \to X) \;\;\; \longmapsto \;\;\; (\varphi_{\ast} \cO_{Y}, \varphi_{\ast}(d_{Y/k}), \mu_{Y}) \;, 
\eeq
from finite coverings to finite locally free $\cO_{X}$-algebras  with a connection and horizontal  multiplication map,  is  functorial \cite[App. E]{AB}. But $\varphi : Y \to X$ is determined by the $\cO_{X}$-algebra $(\varphi_{\ast} \cO_{Y}, \mu_{Y})$ alone.  As a consequence, for any affinoid domain $U \subset X$, any $\cO_{U}$-algebra homomorphism ${\varphi}_{\ast} (\cO_{Y})_{|U} \to \cO_{U}$, is automatically horizontal, hence a solution of $(\cF,\nabla_{\cF})_{|U}$. This proves the first part of the lemma.
\par
 As for the second part of the statement, since two sections of $\varphi_{\ast} \cO_{Y}$ on a connected affinoid domain $U$ coincide as soon as they coincide in the neighborhood of a $k$-rational point of $U$, it will suffice to treat the case of $x \in X(k)$. So, for any point $x_{0} \in X(k) \setminus B$, we consider the completion $\what{\cO}$ of the local ring $\cO_{X,x_{0}}$ and its formal spectrum $\what{X} = \Spf \what{\cO}$;  it is a formal power series ring of the form $k[[t]]$, where $t$ is a local parameter at $x_{0}$, which we may assume to extend to a section of $\cO_{X}$.
 We informally denote by $W \mapsto \what{W}$ the  base-change functor by $\what{X} \to X$ on objects $W$ defined over $X$. It will be enough to prove the statement for the map $\what{\varphi}: \what{Y} \to \what{X}$, at any $x_{0} \in X(k) \setminus B$.
\par Notice that  
$${\cS}ect(Y/X)\,\,\what{} 
 =  {\sH}om_{\cO_{X,x_{0}}-{\rm alg}} (({\varphi}_{\ast} \cO_{Y})_{x_{0}} , \cO_{\what{X}}) =
 {\sH}om_{\cO_{\what{X}}-{\rm alg}} (\what{{\varphi}_{\ast} \cO_{Y}} , \cO_{\what{X}})
\; ,$$ 
is  the set  of formal sections of $\varphi$ at $x_{0}$. 
The $\cO_{\what{X}} $-algebra $\what{\cF} = \what{{\varphi}_{\ast} \cO_{Y}}$ is a direct sum 
$$
 \what{\cF} = \bigoplus_{i=1}^{d} \cO_{\what{X}} \, e_{i}  \; ,
$$
where  
the $e_{i}$'s are orthogonal idempotents. 
An algebra homomorphism 
$\sigma: \what{{\varphi}_{\ast} \cO_{Y}} \to \cO_{\what{X}}$, is forced to map one of the $e_{i}$'s to 1, and the others to 0~: 
let us denote it by $e_{i}^{\ast}$. As we saw before, the $e_{i}^{\ast}$, for $i=1, \dots,d$ are horizontal. Since they freely span the  $\cO_{\what{X}}$-module
$ {\sH}om_{\cO_{\what{X}}} (\what{{\varphi}_{\ast} \cO_{Y}} , \cO_{\what{X}})$ of rank $d$,  they form a $k$-basis of $ {\sH}om_{\cO_{\what{X}}} (\what{{\varphi}_{\ast} \cO_{Y}} , \cO_{\what{X}})^{\nabla}= {\cS}ol(\cF,\nabla_{\cF})_{x}$.
This proves the statement. 
\end{proof}

\begin{parag}
%I am indebted to Qing Liu for the following general statement
%\begin{prop} \label{analpullpush} Let $\varphi:Y \to X$ be a finite morphism of $k$-analytic spaces. Let $\cG$ be any coherent $\cO_{Y}$-module. Then 
%$$
%\varphi^{\ast} \varphi_{\ast} \cG = \varphi^{\ast} \varphi_{\ast} \cO_{Y}  \otimes_{\cO_{Y}} \cG \; .
%$$
%\end{prop}
%The proof is immediate and follows the footprints of the proof of the analogous algebraic  statement, namely
%\begin{prop} \label{algpullpush} Let $\varphi:Y \to X$ be an affine morphism of $k$-schemes. Let $\cG$ be any quasi-coherent $\cO_{Y}$-module. Then 
%$$
%\varphi^{\ast} \varphi_{\ast} \cG = \varphi^{\ast} \varphi_{\ast} \cO_{Y}  \otimes_{\cO_{Y}} \cG \; .
%$$
%\end{prop}
%So, going back to the notation of this paper, we observe that 
%\beq \label{projformula2}
%\varphi^{\ast} \varphi_{\ast}(d_{Y/k}: \cO_{Y} \to \Omega^{1}_{Y}) = \varphi^{\ast} \varphi_{\ast}(d_{Y/k}) :  \varphi^{\ast} \varphi_{\ast} \cO_{Y} \to   \varphi^{\ast} \varphi_{\ast} \cO_{Y}  \otimes \Omega^{1}_Y \; .
%\eeq
%In particular,  
%\begin{cor}
%The inverse image connection
%  \beq \label{singconn}
%  \varphi^{\ast}(\cF,\nabla_{\cF})  =: (\cE,\nabla_{\cE}) \in {\bf MIC}(Y/k) \;  ,
%\eeq
% has no singularity on $Y$. 
% \end{cor}
 We also consider 
the fiber product $Y \times_{X}Y$ and its two projections $\pr_{1} , \pr_{2} : Y \times_{X}Y  \to Y$. 
\end{parag}
We have, as before:
\begin{cor}
 \label{pbsections} The inverse image sheaf $\varphi^{-1}{\cS}ect(Y/X)$ coincides with the sheaf ${\cS}ect(\pr_{1})$ of sections of $\pr_{1}: Y \times_{X}Y \to Y$.
 We have an inclusion of sheaves of sets 
 \beq
 \label{sectionform2}
\varphi^{-1}{\cS}ect(Y/X)   \subset   {\cS}ol(\cE,\nabla_{\cE}) \; .
 \eeq
The sheaf of $k$-vector spaces ${\cS}ol(\cE,\nabla_{\cE})_{|Y \setminus Z}$ is  freely generated by its subsheaf  $\varphi^{-1}{\cS}ect(Y/X)_{|Y \setminus Z}$.
\end{cor} 

 \begin{rmk} 
 The $k$-vector space of $k$-analytic solutions of 
 $(\cE,\nabla_{\cE})$ at any point $z_{0} \in Y(k) \setminus Z$ is spanned by the germs of analytic solutions $w(z)$ at $z = z_{0}$ of the algebraic equation $\varphi (w) = \varphi (z)$. Notice that if $\varphi : \Pan \to \Pan$ is a rational function, the algebraic equation for $w$ as a function of $z$, $\varphi (z+w) = \varphi (z)$ coincides with the equation $w \cdot A_{\varphi}(z,w)=0$ studied by Faber in section 2 of  \cite{Xa2}.
  \end{rmk}

 \begin{parag} 
Let $D$ be a strict open disk in $Y \setminus Z(\varphi)$. Our problem (\ref{RolleANcor})  consists in the determination of the maximal open disk 
 $D_{y_{0}} \subset D$, centered at $y_{0}\in Y(k) \cap D$, such that 
  the map $\varphi$ restricts to an isomorphism 
 \beq \label{lociso1}
 D_{y_{0}} \iso D^{\p}_{ \varphi(y_{0})} \; ,
 \eeq
where $D^{\p}_{ \varphi(y_{0})} $ denotes an open disk with $ \varphi(y_{0}) \in D^{\p}_{ \varphi(y_{0})}  \subset X$. 
 If we set $x_{0} = \varphi(y_{0}) \in X(k) \setminus B$, this problem coincides with the problem of determining the maximal open disk $D^{\p}_{x_{0}}$, centered at $x_{0}$, such that the unique local  section $\sigma$ of $\varphi$ at $x_{0}$ such that $\sigma(x_{0}) = y_{0}$
  converges on  $D^{\p}_{x_{0}}$. By lemma \ref{sections},  $\sigma$ is a
  local solution of $(\cF,\nabla_{\cF})$ at $x_{0}$. 
 Notice that we will then need  to express the result not in terms of  $D^{\p}_{x_{0}}$, but in terms of  
 the height  
 of the annulus  $D \setminus  D_{y_{0}}$ in (\ref{lociso1}), where $ D_{y_{0}} = \sigma(D^{\p}_{x_{0}})$. The statement we want to prove says that
 $$h(D \setminus D_{y_{0}}) \geq p^{-{\frac{1}{p-1}}} \; .$$
 This statement follows if we can prove that the common radius of convergence of $(\cE,\nabla_{\cE})$ at $y_{0}$, expressed in terms of a normalized coordinate on $D$, is $\geq p^{-{\frac{1}{p-1}}}$. 
 \par 
 We will prove that this is the case if $Y$ and $X$ have good reduction. 
 Obviously, in this discussion $Y$ may be replaced by any  strictly affinoid domain with good reduction $C \subset Y$, and $X$ by the image $\varphi(C) \subset X$, provided $D \subset C$ and $\varphi$ induces a finite morphism $C \to \varphi (C)$. 
Moreover, from the discussion of remarks \ref{compactim} and \ref{noncompactim}  it follows that we may assume  that $E:= \varphi(D)$ is a strict open disk in $X$. We will denote by $\zeta$ (resp. $\xi$) the boundary point of $D$ in $Y$ (resp. of $E$ in $X$).
\end{parag}
\section{Basic results on $p$-adic differential systems} \label{basics}
\medskip
In this section $Y$ is any rig-smooth strictly $k$-analytic curve. 
\begin{parag} \label{anelem} The classical theory of $p$-adic linear differential equations is developed on an open disk or an open annulus, embedded as open analytic domains in $\Pan$. Moreover, it is usually understood that their boundary points in $\Pan$ be points of Berkovich type 2. This precision becomes relevant 
when one insists that the coefficients of the equation represent germs of analytic functions at those boundary points.   
One classically defines, for $r \in (0,1) \cap |k|$,
the $k$-Banach algebra $\sH(r,1)$ of  
 \emph{analytic elements} \cite[IV.4]{DGS} on the open annulus  
 \beq 
 \label{annulus} C(0;r,1) := \{x \in D(0,1^{-}) \, |\, r <|T(x)| <1 \, \} \;  .
 \eeq
 It is the completion of the $k$-algebra of rational functions of $T$, with no poles within $C(0;r,1)$, equipped with the sup-norm $||\;||$ on $C(0;r,1)$. While $\sH(r,1) \subset \cB(r,1)$, the Banach $k$-algebra of bounded analytic functions on $C(r,1)$, the two do not coincide, and the properties of 
a first order
 system of linear differential equations 
 \beq \label{diffsyst}
\Sigma  :    \frac{d \, Y}{d\,T} = G\,Y \; ,   
 \eeq
 where $G$ is a $n \times n$ matrix with coefficients in $\sH(r,1)$ are more special than in the case of coefficients in $\cB(r,1)$. 
\end{parag}
\begin{parag} \label{aneldef}
Let $C \subset Y$ be any open analytic domain, with only a finite set $\zeta_{1}, \dots, \zeta_{s}$ of boundary points all of type 2 in $Y$. We define the Banach $k$-algebra $\sH_{Y}(C)$ of  \emph{$Y$-analytic elements} on $C$ as the completion of the $k$-algebra 
$$\cO_{Y}(C) \cap \bigcap_{i=1}^{s}\cO_{Y,\zeta_{i}} \; ,$$
under the sup-norm $||\;||_{C}$ on $C$.   In the present discussion, we consider an open disk $D \subset Y$, with boundary point $\zeta$ of Berkovich type 2. We define  a \emph{formal coordinate} $T$ on $D$ in $Y$, to be a formal \'etale coordinate on (the smooth $\kc$-formal model of) an affinoid domain $A \subset Y$, with good canonical reduction and maximal point $\zeta$, which extends to an isomorphism $T:D \iso D(0,1^{-})$.  A formal coordinate $T$ on $D$ in $Y$ is \emph{overconvergent} if it extends as a section of $\cO_{Y}$ on a neighborhood of $\zeta$ in $Y$. 
 \begin{parag}
As a matter of notation, we recall that, for any   $\kc$-formal scheme $\fX$, locally of finite presentation, with generic fiber the $k$-analytic space $X = \fX_{\eta}$, there is a canonical \emph{specialization map}
\beq
\sp_{\fX}: X = \fX_{\eta}  \to {\fX}_{s} \; ,
\eeq
which may be viewed as a morphism of $G$-ringed spaces 
\beq
\sp_{\fX}: X  = X_{G} \to \fX\; , 
\eeq
where the subscript $(-)_G$ refers to the $G$-topology of \cite[1.3]{BerkovichEtale}. 
\end{parag} 

We now apply comma (3) of proposition 2.2.1 of \cite{BerkInt} to our special case. Notice, contrarily to what is there stated, that the commas (2) and (3) of \lc \, refer to points of type 3 and 2, respectively.  
\begin{lemma} \label{projneigh1} 
Let $Y$ be any rig-smooth strictly $k$-analytic curve and $\zeta$ be a point of type 2 of $Y$.   
Let $C_{\zeta}$ be the union of all open disks  in $Y$ with boundary point $\zeta$. 
Then $C_{\zeta}$ is an affinoid domain in $Y$, with good canonical reduction and maximal point $\zeta$.
Let $\sY$ be a  connected smooth projective curve over $\kc$, such that $C_{\zeta}$ is isomorphic to the complement of a finite number of residue classes in $\sY_{\eta}^{\an}$. Then the embedding $C_{\zeta} \to \sY_{\eta}^{\an}$ extends to an isomorphism 
\beq
U  \iso \sY^{\an}_{\eta} \setminus (\bigcup_{j=1}^{N} ]c_{j}[_{\sY} \cup \bigcup_{i=1}^{n}Y_{i} )\; ,
\label{coord}
\eeq
of an open neighborhood $U$ of $C_{\zeta}$ in $Y$,  where $\zeta$ is sent  to the generic point of $\sY_{s}$ in 
$\sY^{\an}_{\eta}$, $\{c_{1}, \dots, c_{N}\}$ is a well-defined set of  $\kt$-rational points of $\sY_{s}$, for $j=1,\dots,N$  (empty if and only if $\zeta$ is an interior point of $Y$),  $]c_{j}[_{\sY}$ is the residue class of $\sY^{\an}_{\eta}$ corresponding to $c_{j}$, and each $Y_{i} \subset E_{i}$ is  
an affinoid domain isomorphic to the standard closed disk $D(0,1^{+})$ in a residue class $E_{i}$ of $\sY^{\an}_{\eta}$, where $E_{i} \not= E_{j}$, for $i \not= j$.  The isomorphism (\ref{coord}) takes $C_{\zeta}$  to the complement of $\bigcup_{j=1}^{N} ]c_{j}[_{\sY} \cup \bigcup_{i=1}^{n}E_{i} $ in $\sY^{\an}_{\eta}$. 
\end{lemma}
\begin{proof}
The statement easily follows from \lc.
\end{proof}
\begin{cor} 
\label{overcoord} If the strict open disks $D_{1}, \dots,D_{r} \subset Y$  have  the same boundary point $\zeta$ in $Y$, a simultaneous formal overconvergent coordinate on $D_{1}, \dots,D_{r}$ in $Y$ exists. 
\end{cor}
We will also need the following statement.
\begin{prop} \label{locfree} Notation as in lemma \ref{projneigh1}. Let $\cM$ be a locally free $\cO_{Y}$-module of rank $m$, and let $D_{1}, \dots,D_{r} \subset Y$ be strict open disks with the same boundary point $\zeta$ in $Y$. Then, there exists a neighborhood $U$ of $\zeta$ as produced in lemma \ref{projneigh1} , such that $\cM_{|U}$ is free. 
\end{prop}
\begin{proof} We may assume that $Y$ is already of the form of the $U$ in formula  \ref{coord}.
%, and  consider
%an open neighborhood $W$ of $\{\zeta\} \, \dot\cup \,\dot\bigcup_{i} D_{i}$ in $Y$. 
We then pick an affinoid neighborhood $V$ of $\zeta$ in $Y$, such that $\cM_{|V}$ is free. Notice that we may choose $V$ so that 
$\sY_{\eta}^{\an} \setminus V$ is a disjoint union of~:
\ben
\item
the full residue classes $]c_{j}[_{\fY}$, for $j=1,\dots,N$;
\item
some open disks $D^{\p}_{1}, \dots,D^{\p}_{r}$ properly contained in  $D_{1}, \dots,D_{r}$, respectively;
\item
some open disks $E_{i}^{\p}$, with  $Y_{i} \subset E_{i}^{\p} \subset E_{i}$, for $i=1,\dots,n$ (proper inclusions);
\item
a finite number of further open disks $F_{h}$ properly contained in distinct  residue classes $R_{h}$, for $h \in H$. 
\een

We then extend $\cM_{|V}$ to $\sY_{\eta}^{\an}$ as follows. In every residue class $E_{i}$, for $i=1,\dots,n$,  $]c_{j}[_{\fY}$, for $j=1,\dots,N$, or 
$R_{h}$, for $h \in H$, 
we extend $\cM_{|V}$ by direct image, so that 
we obtain a free module on 
$$V \,  \cup \, \l(\bigcup_{i} E_{i}\r) \, \cup \, \l(\bigcup_{j} \, ]c_{j}[_{\fY}\r)  \,  \cup \, \l(\bigcup_{h} F_{h}\r) \; .$$
On the classes $D_{i}$, we extend it by $\cM_{|D_{i}}$, for $i=1,\dots, r$. So, we obtain a locally free $\cO_{\sY_{\eta}^{\an}}$-module $\ol{\cM}$ of rank $m$ extending $\cM$. So, $\ol{\cM} = \cE^{\an}$, for a free $\cO_{\sY_{\eta}}$-module $\cE$ of rank $m$. Then, $\cE$ is free on any affine open subset of $\sY_{\eta}$, and in particular $\cM$ is free on $Y =U$.
%
% a fundamental system  $\{U_{\veps}\}_{\veps \in (0,1)}$  of  open neighborhoods of $\{\zeta\} \, \dot\cup \,\dot\bigcup_{i} D_{i}$ in $Y$, such that for $0< \veps < \veps^{\p} <1$, $U_{\veps} \subset U_{\veps^{\p}}$, and $U \setminus U_{\veps}$ is a disjoint union of a finite family of semi-open annuli. We may assume that $\cM_{|U_{\veps}}$ is free for some $\veps$.  Then $\cM_{|U_{\veps}}$ extends to a free $\cO_{\sY_{\eta}^{\an}}$-module of rank $m$, which descends to a free  algebraic $\cO_{\sX_{\eta}}$-module of rank $m$. 
%
% Let $U = \sY^{\an}_{\eta} \setminus \coprod_{i=1}^{n}Y_{i}$ be as in lemma \ref{projneigh1}, with $D_{i} \subset U$, for all $i$. Let $C_{i}$ be the residue class of $\sY^{\an}_{\eta}$ containing $Y_{i}$. By increasing $Y_{i} \cong D(0,1^{+})$ in $C_{i}$, we may assume that $\cM$ is free over the open annulus 
%$C_{i} \setminus Y_{i}$ (recall that $\cM$ is locally free at $\zeta$). Then we may use a basis of sections of $\cM$ on $C_{i} \setminus Y_{i}$ to extend $\cM$ to the residue class $C_{i}$. In the end we obtain a locally free coherent module $\ol{\cM}$  on  $\sY^{\an}_{\eta}$, which is then of the form $\sN^{\an}$, for a locally free coherent $\cO_{\sY_{\eta}}$-module $\cN$. So, $\cN$ is free on any open affine subset of $\sY_{\eta}$, and \emph{a fortiori} $\cN^{\an}$ is free on $U$. But $\cN^{\an}_{|U} = \cM_{|U}$. 
\end{proof}
 \begin{cor}
 \label{freean} Let  $D \subset Y$ be as in lemma \ref{projneigh1} and let $\zeta$ be the boundary point of $D$ in $Y$. Let $\cM$ be a  locally free $\cO_{Y}$-module of constant rank $m$. Then $\cM(D) \cap \cM_{\zeta}$ is a free module over $\cO_{Y}(D) \cap \cO_{Y,\zeta}$ of rank $m$. For any choice of a basis 
 $\ul{v} := (v_{1},\dots,v_{m})$ of $\cM(D) \cap \cM_{\zeta}$ over $\cO_{Y}(D) \cap \cO_{Y,\zeta}$, let 
$$||
\sum_{i=1}^{m}a_{i}v_{i} ||_{\ul{v}, C} = \max_{i}||a_{i}||_{D} \; ,
$$
be the corresponding norm on $\cM(D) \cap \cM_{\zeta}$.   We define $(\sH^{(\ul{v})}_{Y}(\cM,D), ||\;||_{\ul{v}, C})$ as the completion of 
$\cM(D) \cap \cM_{\zeta}$ under the norm $||\;||_{\ul{v}, C}$. It is a Banach module  over the $k$-Banach algebra $(\sH_{Y}(D),||\;||_{D})$. For two choices $\ul{u}$ and $\ul{v}$ of an $\cO_{Y}(D) \cap \cO_{Y,\zeta}$-basis of $\cM(D) \cap \cM_{\zeta}$, the unique map of $(\sH_{Y}(D),||\;||_{D})$-Banach modules 
$$(\sH^{(\ul{u})}_{Y}(\cM,D), ||\;||_{\ul{u}, C}) \to (\sH^{(\ul{v})}_{Y}(\cM,D), ||\;||_{\ul{v}, C}) \; ,$$
sending $u_{i}$ to $v_{i}$ for all $i$, is a bounded isomorphism. 
\end{cor}
\begin{defn} \label{vecanal} We denote by $(\sH_{Y}(\cM,D), ||\;||_{D})$ any representative of  the uniquely defined isomorphism class of free finitely generated $(\sH_{Y}(D),||\;||_{D})$-Banach modules defined by $(\sH_{Y}(\cM,D), ||\;||_{\ul{v}, C})$, for any choice of $\ul{v}$.
 \end{defn}

Let $T$ be an overconvergent coordinate on $D$ in $Y$. Then, for $C(0;r,1)$ as in (\ref{annulus}),  we define $\sH_{D,Y}(r,1) := \sH_{Y}(T^{-1}(C(0;r,1)))$.  
In the particular case of $D = D(0,1^{-}) \subset \Pan$, with canonical coordinate $T$, $\sH(r,1) = \sH_{D, \Pan}(C(0;r,1))$.  
\end{parag}
We assume that the entries of $G$ in (\ref{diffsyst}) are in $\sH_{D,Y}(r,1)$. Notice that 
$$
Frac(\sH_{Y}(D)) \subset \bigcup_{r < 1} \sH_{D,Y}(r,1) \;,
$$
so that if  the entries of $G$ are quotients of elements of $\sH_{Y}(D)$, then the previous assumption is satisfied for $r<1$, sufficiently close to 1. 
We let $t = T(\zeta) \in \sH(\zeta)$.
    Notice that $g \mapsto |g(\zeta)|$ is a bounded multiplicative norm on $\sH_{D,Y}(r,1)$. 
  The following (almost) classical definition will later be updated. 
 \begin{defn}
 The \emph{generic radius of convergence}  $R_{D \subset Y}(\Sigma)$ of the system $\Sigma$ of (\ref{diffsyst}) on $D\subset Y$, is  defined by extending the field of constants from $k$ to the valued field $\sH(\zeta)$, so that the point $\zeta$  determines a canonical \cite[Intro]{Cont} $\sH(\zeta)$-rational point $\zeta^{\p} \in Y \wt_{k} \sH(\zeta)$, such that $T(\zeta^{\p}) = t$. Notice that the entries of $G$ are  analytic functions on the open disk of $T$-radius 1 in  $Y \wt_{k} \sH(\zeta)$, centered at  $\zeta^{\p}$, so that the  system \ref{diffsyst} is defined on that disk. Then $R_{D \subset Y}(\Sigma)$ is defined  as the $T$-radius  of the maximal open disk around $\zeta^{\p}$, of radius not exceeding 1,  on which all solutions of $\Sigma$ in $\sH(\zeta)[[T - t]]$ converge. 
\end{defn}

\begin{parag}
 The number $R_{D \subset Y}(\Sigma)$ is computed as follows. We first iterate (\ref{diffsyst}) into 
  \beq \label{diffsystenn}
   \frac{1}{n!} \, \frac{d^{n} \, Y}{d\,T^{n}} = G_{[n]}\,Y \; ,   
 \eeq
 and then
 \beq \label{radspec}
 R_{D \subset Y}(\Sigma) = \min (1, \liminf_{i \to \infty}|G_{[i]}(\zeta)|^{-1/i}) = \liminf_{i \to \infty} \max (1, |G_{[i]}(\zeta)|)^{-1/i} 
 \in (0,1]  \; ,
 \eeq
 where the absolute value of a matrix is the maximum of the absolute values of its entries. 
\end{parag}
 The generic radius of convergence of (\ref{diffsyst}) is bounded below as follows  \cite[p. 94]{DGS}.
 \begin{prop}{\bf (Trivial  Estimate)} \label{trivial}
 $$R_{D \subset Y}(\Sigma) \geq \sup(1,|G(\zeta)|)^{-1} \, p^{-\frac{1}{p-1}} \; .$$
 \end{prop}
  \begin{parag}
 We now assume that the entries of the matrix $G$ in (\ref{diffsyst}) extend to meromorphic functions, necessarily with a finite number of zeros and poles, 
 on the open disk $D=D(0,1^{-})$.  
 We also assume that all singularities of the system $\Sigma$ in $D(0,1^{-})$ are \emph{apparent}  \cite[V.5]{DGS}, \ie that at any point $a \in D(0,1^{-})(k)$,  $\Sigma$ admits a matrix solution in $GL(n,k((T-a)))$.
Then the  following  \emph{Transfer Theorem} in a disk with only apparent singularities,  
similar to \cite[IV.5. A]{DGS}, holds.
 \end{parag}
 \begin{thm}{\bf (Transfer Theorem)} \label{transfer} 
 Under the previous assumptions, any solution of $\Sigma$ at any $k$-rational point $x \in D(0,1^{-})$ is meromorphic in a disk of $T$-radius $R_{D \subset Y}(\Sigma)$ around $x$. 
 \end{thm}
 \begin{proof} 
 The point $\zeta$ induces on the field $k(T)$, of rational functions with coefficients in $k$ in the overconvergent coordinate $T$, the classical Gauss norm $|\;|_{Gauss}$ \cite{DGS}. Since the entries of $G$ have a finite number of poles in $D$, 
 we can follow the procedure of  Proposition 5.1 of \cite[Chap. V]{DGS}, to determine a $|\;|_{Gauss}$-unimodular 
 matrix $P \in GL(n,k(T))$, such that $Y \mapsto PY$
  transforms $\Sigma$  into a system 
  $$\Sigma^{[P]}   :    \frac{d \, Y}{d\,T} = G^{[P]}\,Y \; ,   
  $$ 
  with no singularities in $D(0,1^{-})$. Notice that the entries of $G^{[P]}$ are then in $\sH_{Y}(D)$. The matrix $P\in GL(n,k(T))$ is  \emph{$\zeta$-unimodular}, that is  $|P(\zeta)| = |P^{-1}(\zeta)| =1$. It follows that $\max (1, |G^{[P]}_{[i]}(\zeta)|) = \max (1, |G_{[i]}(\zeta)|)$, $\forall i$, and formula \ref{radspec} shows that 
$R_{D \subset Y}(\Sigma^{[P]}) = R_{D \subset Y}(\Sigma)$. We may then assume from the beginning that  $\Sigma$ has no singularities in $D(0,1^{-})$, \ie that the entries of $G$ are in $\sH_{Y}(D)$. 
  But then clearly  $|G_{[i]}(x)| \leq |G_{[i]}(\zeta)| = ||G_{[i]}||_{D}$. A solution matrix of $\Sigma$  at $x$ is given by $Y_{x}(T) = \sum_{n=0}^{\infty} G_{[i]}(x) (T-T(x))^{n}$. So, $Y_{x}(T)$ converges for $|T-T(x)| < R_{D \subset Y}(\Sigma)$.
 \end{proof}

\section{Curves with good reduction}\label{HEART}
\begin{parag} \label{heart}
 We assume here 
 that 
 $\varphi: Y \to X$ is an \'etale covering of  $k$-analytic curves which is the generic fiber of a finite  morphism $\Phi: \fY \to \fX$ of smooth connected $\kc$-formal schemes of relative dimension 1. We observe that if $Y = \Pan$, then $\varphi$ is an isomorphism: we exclude this trivial case. 
 Let $\fF := \Phi_{\ast} \cO_{\fY}$,   a  coherent and locally free $\cO_{\fX}$-module   such that $\fF_{\eta} = \cF$, the $\cO_{X}$-module underlying the connection  (\ref{dirconn}).
\begin{prop} \label{deriv} There exists $\pi_{\Phi} \in \kc$, non-zero, such that
$$
\Phi^{\ast} \Omega^{1}_{\fX/\kc} = \pi_{\Phi} \Omega^{1}_{\fY/\kc} \; .
$$
Therefore
$$
\Phi_{\ast} \Omega^{1}_{\fY/\kc} = \fF \otimes \pi_{\Phi}^{-1} \Omega^{1}_{\fX/\kc} \; .
$$
\end{prop}
\begin{proof} The second statement follows from the first by the projection formula.  Let $\xi_{\fY}$ (resp. $\xi_{\fX}$) be the generic point of $\fY$ (resp. $\fX$), and let $\xi_{Y}$ (resp. $\xi_{X}$) be the maximal point of $Y$ (resp. $X$). The local ring $\cO_{\fY,\xi_{\fY}}$ (resp. $\cO_{\fX,\xi_{\fX}}$) of $\xi_{\fY}$ (resp. $\xi_{\fX}$)  is a valuation ring of rank 1:  its valuation  
extends the one of $\kc$, and has the same value group. We have  $\cO_{\fY,\xi_{\fY}} = \kappa(\xi_{Y})^{\circ}$ (resp. $\cO_{\fX,\xi_{\fX}} = \kappa(\xi_{X})^{\circ}$), hence $k \otimes_{\kc} \cO_{\fY,\xi_{\fY}} = \cO_{Y,\xi_{Y}} =\kappa(\xi_{Y})$ (resp. $k \otimes_{\kc} \cO_{\fX,\xi_{\fX}} = \cO_{X,\xi_{X}} =\kappa(\xi_{X})$). Let $\pi_{\Phi} \in \kc$ be such that 
$$
(\Phi^{\ast} \Omega^{1}_{\fX/\kc})_{\xi_{\fY}} = \pi_{\Phi} (\Omega^{1}_{\fY/\kc})_{\xi_{\fY}} \; .
$$
 Let $E$ be any maximal open disk in $Y$. Since $\varphi(\xi_{Y}) = \xi_{X}$, $E^{\p} := \varphi (E)$ is a maximal open disk in $X$. Let $T$ (resp. $S$)  be  an overconvergent  formal coordinate on $E$ in $Y$ (resp. on $E^{\p}$ in $X$). The map $\varphi$ is then expressed in $E$ by 
$$
S = h(T) \; ,
$$
where $h(T) \in k[[T]]$ is a power series converging and bounded  in $E$, with $||h||_{E} = |h(\xi_{Y})| =1$. Since $\varphi$ is unramifed on $E$, the derivative $dh/dT$ does not vanish on $E$, hence it has a constant absolute value, necessarily equal to $|\pi_{\Phi}|$.
\end{proof}
\begin{cor} \label{appltriv} For any $b \in X(k)$, the connection $(\cF,\nabla_{\cF})$ admits a full set of solutions converging in 
$D_{\fX}(b,p^{-\frac{1}{p-1}}|\pi_{\Phi}|^{-})$.
\end{cor}
\begin{proof}
It is an immediate consequence of the trivial estimate (\ref{trivial}) and of the transfer theorem (\ref{transfer}).
\end{proof}
\begin{rmk}\label{Riemann} The constant $-\frac{1}{p-1} + \ord_{p} \pi_{\Phi}$ may be bound uniformly from below  in terms of $d$ and $p$. See \cite[Thm. 2.1]{Lu}.
\end{rmk}
\begin{cor} \label{Newton} For any $a \in Y(k)$, the map $\varphi$ restricts to an open immersion of $D_{\fY}(a,(p^{-\frac{1}{p-1}})^{-})$ in $X$.
\end{cor}
\begin{proof} We consider a section $\sigma : D_{b}:= D_{\fX}(b,p^{-\frac{1}{p-1}}|\pi_{\Phi}|^{-}) \to Y$ of $\varphi:Y \to X$, and let $a = \sigma(b)$. Then $\sigma(D_{b}) =: D_{a}$, is an open disk in $Y$, $a \in D_{a}$,  and $\varphi$ restricts to an isomorphism $\varphi_{a,b}: D_{a} \iso D_{b}$. Let $E$ (resp. $E^{\p}$), as in the proof of proposition \ref{deriv} be a residue class of $Y$ (resp. $X$) containing $D_{a}$ (resp. $D_{b}$), and let us use the notation of \lc; in particular, we have $|dh/dT(y)| = |\pi_{\Phi}|$, for any $y \in Y$. The $p$-adic Newton lemma \cite[I.4.2]{DGS} implies that, for any $\veps \in (0,1)$, and any $b_{1} \in E^{\p}(k)$, with $|S(b_{1})-S(b)| < \veps |\pi_{\Phi}|^{2}$,  there is a unique $a_{1} \in E(k)$, with $|T(a_{1})-T(a)| < \veps |\pi_{\Phi}|$, such that $\varphi (a_{1}) =b_{1}$. So, for $a_{1}, a_{2} \in D_{a}(k)$, and $b_{1} = \varphi(a_{1}), b_{2} = \varphi(a_{2}) \in D_{b}(k)$, if $|S(b_{1})-S(b_{2})| < |\pi_{\Phi}|^{2}$ and $|T(a_{1})-T(a_{2})| < |\pi_{\Phi}|$, we have 
$$
|T(a_{1})-T(a_{2})|  \leq  |\pi_{\Phi}|^{-1} |S(b_{1})-S(b_{2})| \; .
$$
On the other hand, since $|dh/dT(y)|$ has the constant value  $|\pi_{\Phi}|$ on $E$, there exists $\veps \in (0,1)$, such that, if  $|T(a_{1})-T(a_{2})| < \veps$, then
$$
 |S(b_{1})-S(b_{2})| \leq |\pi_{\Phi}| |T(a_{1})-T(a_{2})|  \; .
$$
In other words,  there exists $\veps \in (0,1)$, such that, if  $|T(a_{1})-T(a_{2})| < \veps$, then
\beq
 |S(b_{1})-S(b_{2})| =  |\pi_{\Phi}| |T(a_{1})-T(a_{2})|  \; .
\eeq
Now, the map $\varphi_{a,b}: D_{a} \iso D_{b}$ being an isomorphism, this estimate must hold for any $a_{1},a_{2} \in D_{a}(k)$, and for their images $b_{1} =\varphi(a_{1})$ and $b_{2} =\varphi(a_{2}) \in D_{b}(k)$ \cite[6.4.4]{Berkovich}. In particular, 
\beq
D_{a} = D_{\fY}(a,(p^{-\frac{1}{p-1}})^{-}) \; .
\eeq
This proves the proposition.
\end{proof}
\par
We have thus concluded the proof of Theorem \ref{RolleAN}.

\end{parag}

\section{Graphs and radius of convergence}
\label{Graphs}
In this section $Y$ is any compact rig-smooth strictly $k$-analytic curve, $\fY$ is a strictly semistable $\kc$-formal scheme \cite[1.1]{Cont} such that $\fY_{\eta} = Y$, and $\fZ$ is a Cartier divisor in $\fY$, finite \'etale over $\Spf \kc$, with generic fiber the divisor $Z \subset Y(k)$. 
\begin{parag}
We recall from \cite{Cont} that to the pair $(\fY,\fZ)$   we can associate a subgraph $\Gamma_{(\fY,\fZ)}$   of the profinite graph $Y$, equipped with a continuous retraction $\tau_{(\fY,\fZ)} : Y \to \Gamma_{(\fY,\fZ)}$. Notice that we are extending the  graph $\Gamma_{(\fY,\fZ)}$ of \cite{Cont}  to include the points of $Z \subset Y(k)$ as vertices ``at infinite distance'' and the retraction $\tau_{(\fY,\fZ)} : Y \to \Gamma_{(\fY,\fZ)}$ by $\tau_{(\fY,\fZ)}(z) = z$, for any $z \in Z$. We do not exclude the case $\fZ= \emptyset$, and we sometimes write $\tau_{\fY} : Y \to \Gamma_{\fY}$ if this is the case. 
The fibers of the retraction $\tau_{(\fY,\fZ)}$ over points of Berkovich type 2 are the closures in $Y$ of the maximal open disks contained in $Y \setminus \Gamma_{(\fY,\fZ)}$. Any such maximal open disk $E$ contains at least a $k$-rational point $x \in Y(k)$; we define $E =: D_{(\fY,\fZ)}(x,1^{-})$. As a $k$-analytic curve, $D_{(\fY,\fZ)}(x,1^{-})$   is isomorphic  to the standard open $k$-disk in $\Pan$, $D(0,1^{-})$, via a \emph{$(\fY,\fZ)$-normalized coordinate at $x$}.  
Given any object $(\cM,\nabla)$ of ${\bf MIC}((Y \setminus  Z)/k)$, and any $x \in Y(k) \setminus Z$, we can define, as in \cite{Cont}, the \emph{$(\fY,\fZ)$-normalized radius of convergence of $(\cM,\nabla)$ at $x$}, 
$\cR_{(\fY,\fZ)}(x,(\cM,\nabla))$, as the radius, measured in $(\fY,\fZ)$-normalized coordinate at $x$, of the maximal open disk $E$ centered at $x$ and contained in $Y \setminus \Gamma_{(\fY,\fZ)}$, such that 
$(\cE,\nabla)_{|E}$ is a free $\cO_{E}$-module of finite rank, equipped with the trivial  connection. 
\end{parag}\begin{parag}
We can also extend the definition of $\cR_{(\fY,\fZ)}(x,(\cM,\nabla))$ to the case when $x \in Y \setminus Z$ is not necessarily $k$-rational.  In full generality, let $K/k$ be a completely valued field extension,  
 let $Y_{K} = Y \wt_{k} K$ and let $\pi_{K/k}: Y_{K} \to Y$, be the projection.  Then there  is a canonical functor \emph{change of field of constants by $K/k$}

 \beq \label{changeconst}
 \begin{array} {cccc}
\pi_{K/k}^{\ast}:& {\bf MIC}((Y \setminus Z)/k) & \to &{\bf MIC}((Y_{K} \setminus Z)/K) \\ \\
& (\cM,\nabla) & \mapsto &  \pi_{K/k}^{\ast}(\cM,\nabla) \; .
 \\
\end{array}
\eeq
 So, let $x \in Y \setminus Z$, not necessarily $k$-rational. 
As in  \cite{Cont}, we change the field of constants by $\sH(x)/k$, and pick (canonically) a $\sH(x)$-rational point $x^{\p} \in Y \wt_{k}\sH(x)$ above $x$. We then set  
 \beq \label{nonratrad}
  \cR_{(\fY,\fZ)}(x,(\cM,\nabla)) := \cR_{(\fY \wt_{\kc} \sH(x)^{\circ},\fZ \wt_{\kc} \sH(x)^{\circ})} (x^{\p}, \pi_{\sH(x)/k}^{\ast}(\cM,\nabla) ) \; .
\eeq 
 This definition is  compatible with any change  of the field of constants by any $K/k$ in the sense that, for any $K/k$ and any $y \in Y_{K} \setminus Z$,
 \beq
 \cR_{(\fY \wt_{\kc}K^{\circ},\fZ\wt_{\kc} K^{\circ})}(y,\pi_{K/k}^{\ast}(\cM,\nabla)) = \cR_{(\fY,\fZ)}(\pi_{K/k}(y),(\cM,\nabla)) \; .
 \eeq
  The function $x \mapsto  \cR_{(\fY,\fZ)}(x,(\cM,\nabla))$ is conjectured to be continuous on $Y \setminus Z$, for any $(\cM,\nabla) \in {\bf MIC}((Y \setminus Z)/k)$. This conjecture was proven in \cite{Cont} under the assumption that $(\cM,\nabla) \in {\bf MIC}_{(\fY,\fZ)}(X (\ast Z)/k)$, \ie that  $\cM$ extends to a locally free coherent $\cO_{\fY}$-module and $\nabla$ has meromorphic singularities at $Z$.
\end{parag}\begin{parag} 
  We now explain the difference between our  radius of convergence  
 $\cR_{(\fY,\fZ)}(x,(\cM,\nabla))$ and the \emph{intrinsic radius of convergence} $IR(\cM_{(x)},\nabla)$ of 
 \beq \label{diffmod}
 (\cM_{(x)},\nabla) := (\cM,\nabla)_{x} 
 \otimes_{\cO_{Y,x}} \sH(x) \; ,
\eeq
 for $x \in Y$ of Berkovich type 2 or 3, of Kedlaya \cite[Def. 9.4.7]{Ke2}. 
 Here $\cO_{Y,x} = \kappa(x)$   is a valued field \cite[2.1]{BerkovichEtale}, $(\cM,\nabla)_{x} $ is a $\kappa(x)/k$-differential module  and $(\cM_{(x)},\nabla)$ is its completion.\footnote{The reader should appreciate the difference between the operation $(\cM,\nabla) \mapsto (\cM_{(x)},\nabla)$, resulting in a $\sH(x)/k$-differential module, and the change of field of constants by $\sH(x)/k$, $(\cM,\nabla) \mapsto \pi^{\ast}_{\sH(x)/k}(\cM,\nabla)$, resulting in an object of 
${\bf MIC}((Y_{\sH(x)} \setminus Z)/\sH(x))$. } Both definitions go back to 
 Dwork and Robba; the latter was refined 
  by Christol-Dwork and used by Christol-Mebkhout and Andr\'e. We will show that two notions coincide at the points $x \in \Gamma_{(\fY,\fZ)} \setminus Z$. 
\end{parag}\begin{parag} 
 Let us shortly review, in our own words, the definition of $IR(\cM_{(x)},\nabla)$, taken from \cite[Chap. 9]{Ke2}.  
Let  $(F,|\;|_{F})/(k,|\;|)$ be a complete extension field. Then $(F,|\;|_{F})$ is a $k$-Banach algebra, and so is $\cL_{k}(F)$, for the operator norm. 
Similarly, on a finite dimensional $F$-vector space $M$, all norms compatible with $|\;|_{F}$ are equivalent and define equivalent structures of $k$-Banach space on $M$. It will be understood in the following that any such $M$ is given some norm of $F$-vector space, compatible with $|\;|_{F}$, and then $\cL_{k}(M)$ is given the corresponding operator norm.  The definitions will be independent of the choices made. 
\par Under the previous assumptions
$\cL_{k}(F)$ (resp. $\cL_{k}(M)$) will  be regarded as an $F$-vector space via the \emph{left} action, $(a \, L)(b) = a\, L(b)$, for $a,b \in F$ (resp. $a \in F$, $b \in M$) and $L \in \cL_{k}(F)$ (resp. $\cL_{k}(M)$) .
\end{parag} 

 \begin{defn} 
A  \emph{complete differential field of dimension 1 over $(k,|\;|)$} is a complete extension field $(F,|\;|_{F})/(k,|\;|)$ such that the $F$-vector space $Der(F/k) \subset \cL_{k}(F)$ of bounded $k$-linear derivations of $F$, is of dimension 1.  
 A \emph{based complete differential field (of dimension 1) over $(k,|\;|)$} is a triple $(F,|\;|_{F},\partial)$ where $(F,|\;|_{F})/(k,|\;|)$ is a complete extension field and $0 \neq \partial \in Der(F/k)$.
 \end{defn}
  \begin{exa}
 A point $x \in \Pan$ of type 2 (resp. 3) is the point $t_{a,\rho}$ at the boundary of the open disk $D(a,\rho^{-})$, for $a \in k$ and $\rho >0$ in $|k|$ (resp. in $\R \setminus |k|$).  One  defines  \cite[Def. 9.4.1]{Ke2}  $F_{a,\rho} = \sH(x)$, as the completion of $k(T)$ under the absolute value 
 $$f(T) \mapsto |f|_{a,\rho} := |f(t_{a,\rho})| \; .$$
 Let $\cL_{k}(F_{a,\rho})$ be the $k$-Banach algebra of bounded $k$-linear endomorphisms of the $k$-Banach algebra $F_{a,\rho}$, equipped with the operator norm. We still denote the operator norm by $|\;|_{a,\rho}$. 
 Then $\frac{d}{dT}$ extends by continuity to  a $k$-derivation of $F_{a,\rho}$, and
 \beq
 |\frac{d}{dT}|_{a,\rho} = \rho^{-1} \; ,
 \eeq
 as an element of $\cL_{k}(F_{a,\rho})$. For the spectral norm of $\frac{d}{dT} \in \cL_{k}(F_{a,\rho})$, we have 
  \beq
 |\frac{d}{dT}|_{{\rm sp}, a,\rho} = p^{-\frac{1}{p-1}} \, \rho^{-1} \; .
 \eeq
 So, the pair (resp. the triple ) $(F_{a,\rho}, |\;|_{a,\rho} )$ (resp. $(F_{a,\rho}, |\;|_{a,\rho} , \frac{d}{dT})$) is a (resp. based) complete differential field of dimension 1 over $(k,|\;|)$. 
\end{exa}
 
 \begin{rmk}\label{diffop}
Let $(F,|\;|_{F})$ be a  complete differential field of dimension 1 over $(k,|\;|)$. Then, for any $F$-basis $\partial$ of $Der(F/k)$ and for any $n \geq 0$, the $F$-vector subspace $Diff^{n}(F/k) \subset \cL_{k}(F)$ of bounded $k$-linear differential operators of $F$ of order $\leq n$, is freely generated by ${\rm id}_{F}, \partial, \dots , \partial^{n}$. 
 \end{rmk}

 \begin{defn} A \emph{finite dimensional differential module over the complete differential field $(F,|\;|_{F})$} (of dimension 1 over $(k,|\;|)$) is a pair $(M,\nabla)$ consisting of a finite dimensional $F$-vector space $M$ and of a $k$-linear bounded $F$-algebra homomorphism
    $$
    \nabla: Diff(F/k) \to \cL_{k}(M)   \; ,
    $$
 such that  
 $$
   \nabla(\partial) (a\,m) = \partial(a)\, m + a \,   \nabla(\partial) (m) \; ,
   $$
   for any $\partial \in Der(F/k)$, $a \in F$ and $m \in M$.  If we specify a generator $\partial$ of $Der(F/k)$ and the corresponding $\Delta = \nabla(\partial)$, we obtain the
  \emph{based} finite dimensional differential module $(M,\Delta)$ over the based complete differential field $(F,|\;|_{F}, \partial)$. 
  \end{defn}
  
    \begin{rmk} Conversely, given a based finite dimensional differential module $(M,\Delta)$ over the based complete differential field $(F,|\;|_{F}, \partial)$, one defines $(M,\nabla)$ by setting
$$
\nabla(\sum_{i=0}^{n}a_{i} \, \partial^{i})= \sum_{i=0}^{n}a_{i} \,\Delta^{i}\; ,
$$
 for any $n$ and any $a_{0}, \dots,a_{n}  \in F$. It is clear that    $\nabla$ is a bounded $F$-algebra homomorphism
    $$
    \nabla: Diff(F/k) \to \cL_{k}(M)   \; . $$
        \end{rmk}
    \begin{defn}
    Let  $(M,\nabla(\partial)) = (M,\Delta)$ be a nonzero finite dimensional based differential module over 
    the based complete  differential field $(F,|\;|_{F}, \partial)$. 
   The \emph{extrinsic  radius of convergence} of  $(M,\Delta)$ is
    $$
    R(M,\Delta) =  p^{-\frac{1}{p-1}} \, |\Delta|_{{\rm sp}}^{-1} > 0\; ,
    $$
    where $ |\Delta|_{{\rm sp}}$ is the spectral norm of $\Delta$ of the $k$-Banach algebra $ \cL_{k}(M)$.
    \end{defn}
        \begin{defn} Let $(M,\nabla) $ be a  finite dimensional differential module  over the  complete differential field $(F,|\;|_{F})$. The \emph{intrinsic  radius of convergence} of  $(M,\nabla)$ is 
    $$
    IR(M,\nabla) = R(M,\nabla(\partial)) \,p^{\frac{1}{p-1}}\, |\partial|_{{\rm sp}} =  |\partial|_{{\rm sp}}  \, |\Delta|_{{\rm sp}}^{-1} \in (0,1]\; ,
    $$
  for any  non zero element $\partial \in Der(F/k)$. 
\end{defn}
   The following proposition explains why $IR(M,\nabla)$ deserves the attribute \emph{intrinsic}. 
\begin{prop} For any $n =0,1, \dots$, let $c_{n} \in \R_{>0}$ be the operator norm of the map of $k$-Banach spaces  
$$\nabla_{n} = \nabla_{| Diff^{n}(F/k)} : Diff^{n}(F/k) \to \cL_{k}(M) \; .
$$
Then 
\beq
\label{spradius}
    IR(M,\nabla) = \liminf_{n \to \infty} c_{n}^{-1/n} \;.
\eeq
\end{prop} 
\begin{proof} Essentially follows from \cite[Prop. 6.3.1]{Ke2}. 
\end{proof}

\begin{cor} \label{normalization} Let $(\cM,\nabla) \in {\bf MIC}((Y \setminus Z)/k)$, as before. Let $\zeta \in Y$ be a point of Berkovich type 2. Let $D$ be any open disk in $Y$ with boundary point $\zeta$,  let $T \in \cO_{Y,\zeta}$ be the germ of a normalized coordinate on $D$
$$T: D \iso D(0,1^{-}) \; .
$$
For any $r \in (0,1) \cap |k|$, let $C(0;r,1)$ be as in (\ref{annulus}).
For $r$ close to 1, we identify the restriction $(\cM,\nabla)_{|C(0;r,1)}$, via the choice of a basis of sections of $\cM$ in a neighborhood of $\zeta$ containing $C(0;r,1)$, with  a differential system $\Sigma$ of the form \ref{diffsyst}. 
Then  
\beq
IR(\cM_{(\zeta)},\nabla) = R_{D \subset Y}(\Sigma)
\; .
\eeq 
\end{cor}
\begin{rmk} We chose a point $\zeta$ of type 2, rather than allowing points of type 3 as well, only in order to avoid extending the field of definition $k$ and to establish contact with the notation of (\ref{annulus}). 
\end{rmk}
\begin{rmk} \label{comparison}
In formula \ref{spradius}, no formal  semistable model $\fY$ of $Y$ explicitly appears. Such a (smooth) model is hidden, however, in the absolute value corresponding to the point $x$ of type 2 or 3. As explained in corollary \ref{normalization}, the normalization of measures at $x$ of type 2 or 3  in this case varies with $x$ and is obtained by taking as an open disk of radius 1, \emph{any open disk with boundary point $x$}.   \end{rmk}
 \begin{parag}
The disadvantage of the function $x \mapsto IR(\cM_{(x)},\nabla)$ which describes the intrinsic radius of convergence of $\cM$ at $x \in Y$ of type 2 or 3,  is that it cannot possibly be extended by continuity to $Y \setminus Z$. In fact, for any point $x_{0} \in Y(k) \setminus Z$, one obviously has 
\beq
 \lim_{x \to x_{0}} IR(\cM_{(x)},\nabla)= 1 \; ,
\eeq
 where the limit runs over the points $x$ of type 2 or 3. But, $Y(k) \setminus Z$ is dense in $Y \setminus Z$, so one would have $R((\cM,\nabla),x) = 1$ identically on $Y \setminus Z$, which is obviously not always the case. The point in our definition of the function $x \mapsto \cR_{(\fY,\fZ)}(x,(\cM,\nabla))$ \cite{Cont} is that 
 \ben
 \item
 it interpolates the classical notion of radius of convergence, normalized by the choice of $(\fY,\fZ)$;
 \item
 it is compatible with extension of the ground field;
 \item
it coincides on the graph $\Gamma_{(\fY,\fZ)}$ with the  intrinsic radius of convergence
\beq
 IR(\cM_{(x)},\nabla) = \cR_{(\fY,\fZ)}(x,(\cM,\nabla)) \; , \;\; \hbox{if} \;\; x \in  \Gamma_{(\fY,\fZ)} \; .
 \eeq
 \een
 The last property follows from remark \ref{comparison}.
 \end{parag}

 %%%%%%%%%%%%%%%%%%%%%%%%%%%%%%%%%%%%%%%%%%%%%%%%%%%%%%%%%%%%%%%%%%%%
%%%%%%%%%%%%%%%%%%%%%%%%%%%%%%%%%%%%%%%%%%%%%%%%%%%%%%%%%%%%%%%%%%%%%%%%%%%%%%%%%%%%%%%%

\end{document}